\documentclass[11pt,letterpaper]{amsart}
\usepackage{amssymb, amscd, amsfonts, euler, epic, xy}
\xyoption{all}
\usepackage{charter}

\def\CC{{\mathbb C}}
\def\DD{{\mathbb D}}

\def\PP{{\mathbb P}}
\def\QQ{{\mathbb Q}}
\def\RR{{\mathbb R}}

\def\ZZ{{\mathbb Z}}

\def\G{{\Gamma}}
\def\g{{\gamma}}

\def\an{{\rm an}}
\def\bb{{\rm bb}}

\def\reg{{\rm reg}}

\def\sing{{\rm sg}}

\def\bb{{bb}}

\def\bs{\backslash}

\def\pt{{\bullet}}
\def\eps{\epsilon}
\def\Mcalint{\dot{\Mcal}}
\def\Xint{\dot{X}}
\def\Dint{\dot{\DD}}
\def\Tint{\dot{T}}
\def\Gint{\dot{G}}

\def\Acal{{\mathcal A}}
\def\Bcal{{\mathcal B}}

\def\Hcal{{\mathcal H}}

\def\Mcal{{\mathcal M}}

\def\Ocal{{\mathcal O}}

\def\Scal{{\mathcal S}}

\def\Ycal{{\mathcal Y}}

\def\la{\langle}
\def\ra{\rangle}

\newcommand\aut{\operatorname{Aut}}

\newcommand\orth{\operatorname{O}}
\newcommand\proj{\operatorname{Proj}}

\newcommand\rk{\operatorname{rk}}
\newcommand\sym{\operatorname{Sym}}
\newcommand\Star{\operatorname{Star}}

\newcommand\GL{\operatorname{GL}}

\newcommand\PGL{\operatorname{PGL}}

\newcommand\SL{\operatorname{SL}}

\newcommand\Orth{\operatorname{O}}

\newcommand\PSL{\operatorname{PSL}}

\newtheorem{theorem}{Theorem}[section]
\newtheorem{lemma}[theorem]{Lemma}
\newtheorem{proposition}[theorem]{Proposition}

\newtheorem{corollary}[theorem]{Corollary}

\theoremstyle{definition}

\theoremstyle{remark}
\newtheorem{remark}[theorem]{Remark}

\newtheorem{properties}[theorem]{Properties}

\begin{document}
\title{The period map for cubic fourfolds}
\author{Eduard Looijenga}
\email{looijeng@math.uu.nl}
\address{Mathematisch Instituut\\
Universiteit Utrecht\\ P.O.~Box 80.010, NL-3508 TA Utrecht\\
Nederland}

\begin{abstract} 
The period map for cubic fourfolds takes values in a  locally symmetric variety of orthogonal type of dimension 20. We determine the image of this period map (thus confirming a conjecture of Hassett) and give at the same time
a new proof of the theorem of Voisin that asserts that this period map  is an open embedding.  An algebraic version of our main result is an identification of the algebra of $\SL (6,\CC)$-invariant polynomials on the representation space $\sym^3(\CC^6)^*$ with a certain algebra of meromorphic automorphic forms on a symmetric domain of orthogonal type of dimension  $20$. We also describe the stratification of the
moduli space of semistable cubic fourfolds in terms of a Dynkin-Vinberg diagram.
\end{abstract}
\maketitle 
\section*{Introduction}
The primitive cohomology of nonsingular cubic fourfold $Y\subset \PP^5$ is located in the middle
dimension (four) and has as nonzero Hodge numbers $h^{3,1}=h^{1,3}=1$ and
$h_o^{2,2}=20$. If we make a Tate twist (which subtracts $(1,1)$ from the bidegrees),
then this  looks very much like the primitive cohomology  of a polarized K3 surface, the difference only being that  the $(1,1)$ summand is of dimension $20$ instead of $19$. This observation has in a sense been explained by Arnaud Beauville and Ron Donagi \cite{bd}: they showed that that the Fano variety of lines on $Y$ is a deformation of the  symmetric square with resolved diagonal of a polarized  K3 surface (this Fano variety is an example of a complex symplectic fourfold). It was the point of departure for Claire Voisin \cite{voisin} for her proof of the injectivity of the period map for  cubic fourfolds (which amounts to the assertion  that the polarized Hodge structure on the primitive cohomology of a nonsingular cubic fourfold determines the fourfold up to projective transformation). 

The question that remained was the image of this period map. If we make the passage from the cubic fourfold to its Fano variety, then a theorem of Huybrechts 
\cite{huybrechts}  asserts that we have essentially surjectivity: every Hodge structure of this type can be realized
by a deformation of a Fano variety of a cubic fourfold. However, it is not at all clear that any global deformation of a  Fano variety of a cubic fourfold comes from a 
deformation of the fourfold. Indeed, it was suspected by Brendan Hassett that this is not the case and in a letter to the author (dated March 22, 2002) he conjectured that the image of the
period map for cubic fourfolds (with innocent singularities allowed) would miss just
(what we call) an arithmetic arrangement. The missing Hodge structures
would be `swallowed' by the secant variety of the Veronese surface in $\PP^5$ in the sense that they only appear as the limiting Hodge structures  of all possible smoothings of this variety.

The main goal of this paper is to prove Hassett's conjecture. But our proof yields more, such as a 
new proof of Voisin's injectivity theorem.  We also find a Vinberg-Dynkin diagram of an arithmetic 
reflection group of hyperbolic type of rank 20 that gives an insightful picture of the boundary strata 
and their incidence relations.  (It is the analogue of a similar diagram of a rank 19 arithmetic 
reflection group that we obtained long ago for 
K3 surfaces of degree 2 and that is described in  \cite{scattone}, p.\ 82ff.) 

The proof uses in a fundamental way the techniques and results that we developed in 
\cite{looijenga} and  (jointly  with Swierstra) in \cite{ls}, with applications as the
present one in mind. These  pertain to compactifications of varieties of the form `locally symmetric  variety minus a locally symmetric hypersurface' and associated algebra's of meromorphic automorphic forms.  They are powerful enough to enable us to identify certain GIT-compatifications without detailed knowledge of that
compactification and their geometry. We initially used as our GIT-input the
(as yet unpublished)  work by Mutsumi Yokoyama \cite{yokoyama}, but recently  a more detailed  classification, due to Radu Laza \cite{laza}, has become available, that allowed us to shorten some of our  arguments. He has recently used  used his GIT analysis to
give an alternate proof of Hassett's conjecture along the lines of Shah's approach to K3 
surfaces of degree two \cite{laza2}.
\\

Let us now briefly comment on the contents  of the individual sections.
Section 1 is mostly a study of a lattice abstractly isomorphic to the primitive cohomology of a cubic fourfold.  We find  an arithmetic reflection group in a hyperbolic lattice of rank 20 that, among other things,  yields a classification of primitive isotropic sublattices. 
  
We use these results to describe in Section 2 a certain locally symmetric variety
of orthogonal type of dimension 20,
a locally symmetric hypersurface in this variety, and a compactification of its complement (which we later identify with the moduli space of semistable cubic fourfolds).

In Section 3 we define the period map and state our principal result.

Section 4 reviews our (rather elementary) theory of boundary pairs in a manner
that is adapted to the present situation.

Section 5 consists of  computing the degree four homology of the smooth part
of the two most singular semistable cubic fourfolds: the secant variety of the Veronese
surface and the one defined by $Z_0Z_1Z_2=Z_3Z_4Z_5$. 

In Section 6 we prove our principal result. The proof is relatively short and could have been shorter still had we not wished to include an alternative proof of Voisin's 
injectivity theorem. For the latter purpose we need  to study in some detail the  automorphism group and the deformation theory of  a cubic fourfold  of the form
$Z_0Z_1Z_2=\Phi (Z_3,Z_4,Z_5)$, where $\Phi$ defines a nonsingular plane cubic
(this fourfold has three singular points, each of type $\tilde E_6$). This is done in  
Section 7. 
\\

Although this paper uses much of the techniques developed in our earlier papers, we tried to make not all of these a prerequisite. Some familiarity with 
\cite{looijenga} remains indispensable however.

\section{The primitive cohomology lattice of cubic fourfolds}
Let $\Lambda$ be an odd unimodular  lattice of signature $(21,2)$ and 
$\eta\in\Lambda$ such that $\eta\cdot\eta=3$ and the orthogonal complement 
$\Lambda_o$ of $\eta$ is even. We denote by $\G$ resp.\   $\hat\G$ the stabilizer of 
$\eta$ resp.\  of $\ZZ\eta$  (or equivalently, $\Lambda_o$) in the 
orthogonal group of $\Lambda$. Since minus the identity is in $\hat\G-\G$, we have
$\hat\G=\{ \pm1\}\times\G$. If an element of $\hat\G$ acts trivially on $\Lambda_o$, then
it will leave $\eta$ fixed (for we have a natural identification of $\ZZ\eta/(3\eta)$ with
$\Lambda_o^*/\Lambda_o$), and so $\hat\G$ acts faithfully on $\Lambda_o$.
We may characterize $\hat\G$  as the full orthogonal group of $\Lambda_o$ and
$\G$ as the subgroup that acts trivially on $\Lambda_o^*/\Lambda_o$.
For reasons that become clear shortly, we will call a vector $v\in\Lambda_o$ 
a \emph{long root} if $v\cdot v=2$; such a vector has the property that the orthogonal reflection in it, $s_v: x\mapsto x-(x\cdot v)v$, is in $\G$.

We can identify $\Lambda$ with $2E_8\perp 2U\perp 3I$ (here $I$ denotes the odd unimodular rank one lattice: it has  a generator $\eps$ with $\eps\cdot\eps=1$) in such a manner that  $\eta=\eps_1+\eps_2+\eps_3$. So then  $\Lambda_o=2E_8\perp 2U\perp A_2$, where
$A_2$ is spanned by $\beta_1:=\eps_1-\eps_2$ and $\beta_2:=\eps_2-\eps_3$.
Notice that the orthogonal complement of $\beta_2$ in $A_2$ is spanned  by 
$h_1:=-2\eps_1+\eps_2+\eps_3=\eta-3\eps_1$, a vector with the property that  (i) $h_1\cdot h_1=6$ and (ii) $\eta-h_1\in 3\Lambda$. Indeed, the primitive hull of the span of 
$h_1$ and $\eta$ is the lattice spanned by $\eps_1$ and $\eps_2+\eps_3$, hence of
type $A_1\perp I$ and the orthogonal complement of this lattice is 
$2E_8\perp 2U\perp A_1$. 

\begin{lemma} For a vector $h\in\Lambda_o$ with $h\cdot h=6$  we have
$h\in \G h_1$ resp.\ $h\in\hat\G h_1$  if and only if $\eta-h$ is divisible by $3$ resp.\ 
$\frac{1}{3}h$ generates $\Lambda_o^*/\Lambda_o$; in these cases we say that $h$ is \emph{special} resp.\  that $\frac{1}{3}h$ is a \emph{short root}. The orthogonal reflection in a short root preserves $\Lambda_o$ and hence acts as an element of $\hat\G$.
\end{lemma}
\begin{proof}
If $\eta-h$ is divisible by $3$: $\eta-h=3\eps$, then, as we have seen in the above argument, $\eta-\eps$ and $\eps$ are perpendicular vectors of square norm $2$ and $1$ respectively.
Since $\Lambda$ contains two summands of type $U$,  there is (according to 
well-known result in lattice theory) a  
$g\in \Orth (\Lambda)$ such that $g(\eps)=\eps_1$. Then $g(\eta-h)\in 2E_8\perp 2U\perp 2I$ and for the same reason as above (the occurrence of $2U$), there is an orthogonal transformation $g'$ of this lattice (that we think of as an orthogonal transformation fixing $\eps_1$) that sends $g'(\eta-h)$ to $\eta-h_1$. So $g'g$ is an element of $\G$ that sends $h$ to $h_1$. This implies all the assertions of the lemma, except the last. But for that we observe that
$s_{h}: x\in\Lambda_o\mapsto x-\frac{1}{3}(x\cdot h)h$ is evidently orthogonal and preserves $\Lambda_o$ and hence lies in $\hat\G$.
\end{proof}

We denote the set of special vectors by $\Hcal$. It is clear that for a short root $r$, either $3r$ or $-3r$ is special and that $r\cdot r=\frac{2}{3}$.

The vectors  $h_i:=\eta-3\eps_i$ ($i=1,2,3$) are all special, have zero sum
and lie in the $A_2$-summand of $\Lambda_o$. Since the short
and the long roots in that summand make up a root system of type $G_2$, we call it a \emph{$G_2$-summand}.

\begin{lemma}\label{lemma:specialpos}
A set of special vectors  spans a positive definite sublattice if and only if it is contained in 
a $\G$-translate of $\{h_1,h_2,h_3\}$.
\end{lemma}
\begin{proof}
Let $h,h'\in\Hcal$ be distinct and span a positive definite sublattice. 
Since $h,h'\in\Hcal$ span a positive definite lattice of rank two we must have $|h\cdot h'|<6$. Since  $h\cdot h'$ is divisible by $3$, it must therefore lie in $\{-3,0,3\}$. We know that $h'=h+3v$ for some $v\in \Lambda_o$.  From $h\cdot h'\in \{-3,0,3\}$ we get that $h\cdot v\in\{-1,-2,-3\}$. Since $h'\cdot h'=6$ we also have $2h\cdot v+3v\cdot v=0$, and so
the only possibility is that $h\cdot v=-3$ and $v\cdot v=2$. This implies $h\cdot h'=-3$.
If there is a third element $h''\in\Hcal -\{ h,h'\}$ such that the span
of $h,h',h''$ is positive definite, then we have also $h''\cdot h=h''\cdot h=-3$. This implies that $h+h'+h''$ is isotropic.
As we assumed that the lattice spanned by $h,h',h''$ is positive definite, it follows that $h+h'+h''=0$. 
So we have then a maximal subset of $\Hcal$ that spans a postive definite sublattice.
Returning to the pair   $h,h'\in\Hcal$, then $\eps':=\frac{1}{3}(\eta-h')$, $\eps'':=\frac{1}{3}(\eta-h'')$ and
$\eps''':=-\eps-\eps'+\eta$ span mutually perpendicular vectors of norm $1$. We thus get an embedding  $j:3I\to \Lambda$
that sends the sum of the generators to $\eta$.
The orthogonal complement of $j$  is even, unimodular and of signature $(18,2)$, hence isomorphic to $2E_8\perp 2U$.
Any  vector in $3I$ of selfproduct $3$ is a signed sum of basis vectors and hence equivalent to $\eta$. This implies
that $j$ may be composed with an element of $\G$ to produce the given embedding of $3I$ in $\Lambda$. This will
take $h'$ and $h''$ to $\{h_1,h_2,h_3\}$.
\end{proof}

\begin{lemma}\label{lemma:iii}
The primitive isotropic elements of $\Lambda_o$ lie in a single $\G$-orbit.
\end{lemma}
\begin{proof} This is again a formal consequence of  the fact that $\Lambda_o$ contains a sublattice isomorphic to $2U$.
\end{proof}

\subsection*{The associated hyperbolic lattice} An example of a primitive isotropic vector is the basis vector $e_2$ in the second hyperbolic summand. It is clear that $e_2^\perp/\ZZ e_2=2E_8\perp A_2\perp U$. This lattice, that we shall denote  by $\Lambda_1$, has hyperbolic signature. Let $W$ denote the group generated by reflections in all the roots (long and short)  in $\Lambda_1^*$.  If we implement Vinberg's algorithm for finding a fundamental polyhedron of $W$, we find that it terminates. We thus end up with a \emph{finite} collection $B$ of roots such that the inner product between any pair of distinct element  is $\le 0$ and that every root in $\Lambda_1^*$ is a linear combination
of elements of $B$ with all coefficients in $\ZZ_{\ge 0}$ or all in $\ZZ_{\le 0}$.
This is equivalent to the Dynkin diagram $D(B)$ having the property that the vertices of every \emph{maximal} subdiagram  of finite resp.\ of  \emph{pure} affine type (meaning that  all connected components of its are of that  type) span a sublattice of corank one in $\Lambda_1^*$. 

The diagram $D(B)$ is  conveniently described in abstract terms. Let us first do this for the full subdiagram $D(B_\ell)$ on the set of long roots $B_\ell$ in $B$. 
For this we begin telling how the elements of $B_\ell$ can be found in the 
lattice $\Lambda_1=2E_8\perp A_2\perp U$. Denote the fundamental roots 
of $E_8$ by $(\alpha_1,\dots ,\alpha_8)$ and  denote by $(\varpi_1,\dots ,\varpi_8)$
its dual basis of fundamental weights. For the corresponding elements in the second  copy of $E_8$ we use a prime. Recall that we had already introduced the  root basis 
$(\beta_1=\eps_1-\eps_2,\beta_2=\eps_2-\eps_3)$ of $A_2$. We put $\beta_0:=-\beta_1-\beta_2=\eps_3-\eps_1$.
Now let $B_\ell$ consist of the long roots
$\{\alpha_1,\dots ,\alpha_8,-\varpi_8-e\}$ (of type $\hat E_8$), 
$\{\alpha'_1,\dots ,\alpha'_8,-\varpi'_8-e\}$ (also of type $\hat E_8$), 
$e+f$, $-e+\beta_0$, $\beta_1$, $-\varpi_1-\varpi'_1-2e+2f+\beta_0$,
$-\varpi_2-\varpi'_7-3e+3f-\beta_1-2\beta_2$, $-\varpi_7-\varpi'_2-3e+3f+\beta_0$.
Notice that any two elements of $B_\ell$ have inner product $0$ or $-1$. The  corresponding Dynkin graph $D(B_\ell)$ is as in Figure \ref{fig1}:
where (say) $u=e+f$, $ua=-e-\beta_1-\beta_2$, $au=\beta_1$, $a=-\varpi_1-\varpi'_1-2e+2f+\beta_0$, $bv=-\varpi_2-\varpi'_7-3e+3f+\beta_0-\beta_2$,
$cw=-\varpi_7-\varpi'_2-3e+3f+\beta_0-\beta_2$ (this funny labeling will become clear in a moment), and the 
remaining vertices produce the two copies of $\hat E_8$.  This picture reveals a symmetry that was not apparent before (and suggested our new labeling of the vertices): we can describe the abstract graph 
$D(B_\ell)$ solely  in terms of  a  $6$-element set $\Bcal$ of which is given a partition
into $3$-element subsets $\Bcal', \Bcal''$ ($\Bcal$ is the set of branch points of $B_\ell$
partitioned by the equivalence relation of being not connected by an edge):  consider their join $\Bcal'\star \Bcal''$ (a graph whose set of vertices is $\Bcal$ and whose set of edges is the set of unordered pairs, one item in $\Bcal'$, another in $\Bcal''$). Then $D(B_\ell)$ is obtained by  putting on each edge of this join two additional vertices.  So any new vertex has as its label an element of $(\Bcal'\times \Bcal'')\cup  (\Bcal''\times \Bcal')$. This is illustrated by Figure \ref{fig1}, where we have denoted the elements of the two sets by $\{a,b,c\}$ and
$\{u,v,w\}$. So the vertices  of degree $3$ of $B_\ell$ may be denoted 
$r_a,\dots ,r_w$ and those of degree $2$ by $r_{au},\dots ,r_{wc}$. It is clear that the
automorphism group of $B_\ell$ can be identified with the group of permutations
of $\Bcal$ that preserve the decomposition (which is a semidirect product
$\ZZ/2\ltimes (\aut(\Bcal')\times\aut (\Bcal''))$).

In $B_\ell$ we recognize the following maximal  subdiagrams of pure affine type:
\begin{enumerate}
\item[$3\hat E_6$:] omit $\Bcal'$ or $\Bcal''$, 
\item[$\hat D_7\perp\hat A_{11}$:] omit two 
degree two vertices at distance $5$, 
\item[$A_{17}$:] omit the interior vertices on three disjoint lines of the join, 
\item[$\hat E_7\perp \hat D_{10}$:] omit for instance $\{au,bv,bw, cv,cw\}$, 
\item[$\hat D_{16}$:] omit for instance the strings $\{ a,au,au\}$, $\{bv,vb\}$,
$\{cw,wc\}$, 
\item[$2\hat E_8$:] omit for instance the  string  $\{a,au,ua,u\}$ and the vertices
$\{bv,cw\}$.
\end{enumerate}
Only in the first two cases the corresponding subset of $B_\ell$ spans  a sublattice of corank one. For this reason we need some short roots to produce $B$. 
These will then produce 
an extra affine summand of type $\hat A^s_1$ (consisting of short roots, this is what the superscript $s$ stands for) or $\hat G_2$
so that now the corank one property is fulfilled in all cases. The 
short roots $B_s$ in $B$ that we find are indexed the set 
of bijections from one part of $\Bcal$  onto the other (so from $\Bcal'$ onto
$\Bcal''$ or vice versa): for such a bijection $\sigma$  the corresponding short root $r_\sigma$ has the following properties:
$r_\sigma$ has inner product zero with any root of $B_\ell$ unless it is a degree two vertex of the form  $r_{b\sigma (b)}$ for which $b$ is in the domain of $\sigma$. Furthermore, the inner product between two short roots is as follows:
\[
r_\sigma\cdot r_\tau=
\begin{cases}
-\frac{2}{3} &\text{ if $\tau=\sigma^{-1}$ or $\sigma\tau^{-1}$ exists and has order $2$,}\\
-\frac{4}{3} &\text{ if $\sigma\tau$ exists and is of order $2$ or $\sigma\tau^{-1}$ exists and has order $3$,}\\
-\frac{5}{3} &\text{ if $\sigma\tau$ exists and has order $3$.}
\end{cases}
\]
In the first case, $r_\sigma-r_\tau$ is isotropic and $\{\sigma,\tau\}$ defines a $\hat A_1$-subdiagram consisting of short roots (that can be extends in $B$ in several ways to a $\hat G_2$-diagram). In the other two cases,  $r_\sigma,r_\tau$ generate a sublattice of hyperbolic signature so that $\{\sigma,\tau\}$ is of hyperbolic type.

\medskip
{\begin{picture}(50, 250)%
\linethickness{1pt}

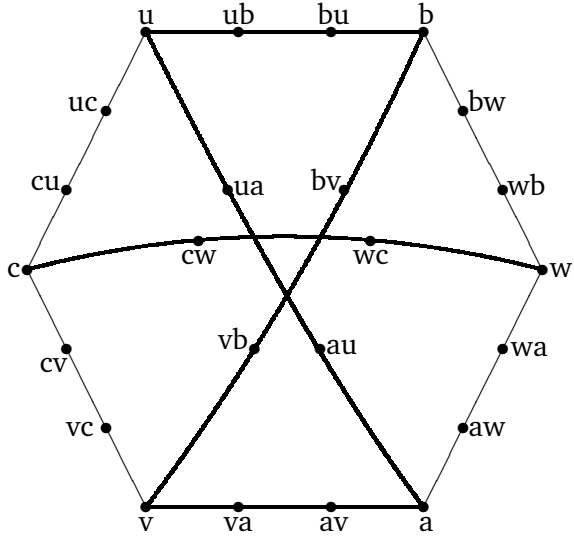
  \put(180,235){\makebox (0,0)[b]{Fig. 1: The long roots diagram attached to $2E_8\perp A^*_2\perp U$}}\label{fig1}
   \put(70, 40){\circle*{4}}%
   \put(70, 40){\line(1, 0){35}}%
   \put(105, 40){\circle*{4}}%
   \put(105, 40){\line(1, 0){35}}%
   \put(140, 40){\circle*{4}}%
   \put(140, 40){\line(1, 0){35}}%
   \put(175, 40){\circle*{4}}%
    
    \put(70,31){\makebox(0,0)[b]{v}}%
    \put(105,31){\makebox(0,0)[b]{va}}%
    \put(141,31){\makebox (0,0)[b]{av}}
    \put(176,31){\makebox (0,0)[b]{a}}
  
   \put(45,67){\makebox (0,0)[b]{vc}}
   \put(35,91){\makebox (0,0)[b]{cv}}
   \put(20,127){\makebox (0,0)[b]{c}}
   
   \put(32,160){\makebox (0,0)[b]{cu}}
   \put(46,190){\makebox (0,0)[b]{uc}}
   \put(70,223){\makebox (0,0)[b]{u}}
  
   \put(105,223){\makebox (0,0)[b]{ub}}
   \put(141,223){\makebox (0,0)[b]{bu}}
    \put(176,223){\makebox (0,0)[b]{b}}
  
   \put(199,67){\makebox (0,0)[b]{aw}}
   \put(215,97){\makebox (0,0)[b]{wa}}
   \put(227,127){\makebox (0,0)[b]{w}}
   \put(214,158){\makebox (0,0)[b]{wb}}
   \put(199,190){\makebox (0,0)[b]{bw}}
 
   \put(70, 40){\line(-1, 2){15}}%
   \put(55, 70){\circle*{4}}%
   \put(55, 70){\line(-1, 2){15}}%
   \put(40, 100){\circle*{4}}%
   \put(40, 100){\line(-1, 2){15}}%
   \put(25, 130){\circle*{4}}%

   \put(25, 130){\line(1, 2){15}}%
   \put(40, 160){\circle*{4}}%
   \put(40, 160){\line(1, 2){15}}%
   \put(55, 190){\circle*{4}}%
   \put(55, 190){\line(1, 2){15}}%
   \put(70, 220){\circle*{4}}%
   
   \put(70, 220){\line(1, 0){35}}%
   \put(105, 220){\circle*{4}}%
   \put(105, 220){\line(1, 0){35}}%
   \put(140, 220){\circle*{4}}%
   \put(140, 220){\line(1, 0){35}}%
   \put(175, 220){\circle*{4}}%

   \put(175, 40){\line(1, 2){15}}%
   \put(190, 70){\circle*{4}}%
   \put(190, 70){\line(1, 2){15}}%
   \put(205, 100){\circle*{4}}%
   \put(205, 100){\line(1, 2){15}}%
   \put(220, 130){\circle*{4}}%

   \put(220, 130){\line(-1, 2){15}}%
   \put(205, 160){\circle*{4}}%
   \put(205, 160){\line(-1, 2){15}}%
   \put(190, 190){\circle*{4}}%
   \put(190, 190){\line(-1, 2){15}}%
   
   \qbezier(25,130),(122, 155),(220,130)
   \put(90, 141){\circle*{4}}%
   \put(90,132){\makebox (0,0)[b]{cw}}
   \put(155,141){\circle*{4}}%
   \put(155,132){\makebox (0,0)[b]{wc}}
   
   \qbezier(70,40),(130, 120),(175,220)
   \put(111, 100){\circle*{4}}%
   \put(103,100){\makebox (0,0)[b]{vb}}
   \put(145,160){\circle*{4}}%
   \put(138,160){\makebox (0,0)[b]{bv}}
   
   \qbezier(70,220),(130, 100),(175,40)
   \put(101,160){\circle*{4}}%
   \put(109,158){\makebox (0,0)[b]{ua}}
   \put(136,100){\circle*{4}}%
    \put(144,98){\makebox (0,0)[b]{au}}
    
\end{picture}}

\begin{properties}\label{properties}
It is known that 
\begin{enumerate}
\item The  open polyhedral cone 
$C\subset\Lambda_1\otimes\RR$ defined by $x\cdot b<0$ for all 
$b\in B$ is contained in a connected component $(\Lambda_1\otimes\RR)_+$ of
the set $x\in \Lambda_1\otimes\RR$ with $x\cdot x<0$. 
\item $C$ is a connected component in the set of $x\in \Lambda_1\otimes\RR$ that
lie on no reflection hyperplane.
\item $W\overline{C}$  is the convex hull of 
$\overline{(\Lambda_1\otimes\RR)_+}\cap \Lambda_1$ (this is also the union of $(\Lambda_1\otimes\RR)_+$ and the rays spanned by an isotropic vector in $\Lambda_1$ on the boundary of $(\Lambda_1\otimes\RR)_+$) and  $\overline{C}$  is a strict fundamental domain for the action of $W$ in $\overline{(\Lambda_1\otimes\RR)_+}$. Moreover, $W$ is generated by the the reflections in the elements of $B$.
\item The index two subgroup of the orthogonal group of $\Lambda_1$ that preserves $(\Lambda_1\otimes\RR)_+$ contains $W$ as a normal subgroup with quotient the symmetry group $\aut(D(B))=\ZZ/2\ltimes (\aut(\Bcal')\times\aut (\Bcal''))$.
\item Let $C\subset C_s\subset (\Lambda_1\otimes\RR)_+$ denote the connected component of the set of 
$x\in (\Lambda_1\otimes\RR)_+$ that lie on no reflection hyperplane of a short root and contains $C$.  Then the 
$W$-stabilizer of $C_s$ is the subgroup $W(B_\ell)$  of $W$ generated by the reflections in  the set of long roots in $B$, $B_\ell$. 
\end{enumerate}
\end{properties}

\begin{corollary}\label{cor:ii}
The orbits of the primitive isotropic vectors in $\Lambda_1$ under the orthogonal group
of $\Lambda_1$ are separated by the root system in their stabilizer: they are of affine type $2\hat E_8\perp \hat G_2$, $\hat D_{16}\perp \hat G_2$, $\hat A_{17}\perp \hat A_1^s$, $\hat E_7\perp \hat D_{10}\perp \hat A_1^s$, $3\hat E_6$, $\hat D_7\perp \hat A_{11}$. (Here $A_1^s$ stands for a copy of $A_1$ spanned by a short root.)

The $\G$-orbits of the primitive isotropic lattices $K\subset\Lambda_o$ of rank $2$ can be distinguished by the isomorphism type of the associated positive definite even lattice $K^\perp/K$. These, in turn, can be distinguished by the root systems without long roots that they contain and the types that thus appear are  $2E_8$, $D_{16}$, $A_{17}$,  $E_7\perp D_{10}$, $3E_6$ and $D_7\perp A_{11}$.
\end{corollary}
\begin{proof}
We observe that the maximal subdiagrams of $D(B)$ of pure affine type are
precisely the affine completions of the root systems listed and that those
of a given type lie in a single $\aut (B)$-orbit. In view of the Properties \ref{properties} above it follows that the orbits of the orthogonal group of $\Lambda_1$ in the set of primitive isotropic vectors are in bijection of the root system types listed. 

Now let $K$ be an isotropic plane in $\Lambda_o$. Since  its  $\G$-orbit  is also
a $\hat\G$-orbit it is enough to show that some element of  $\hat\G$.
Choose a primitive vector in $K$. Then there exists an element of $\G$ that takes that element to $e_2$. We may therefore assume that $e_2\in K$. Then $K/\ZZ e_2$ defines
a primitive isotropic rank one lattice in $\Lambda_1$ and the assertion follows from the discussion above.
\end{proof}

\section{The arithmetic arrangement}

We have a quadric $\check{\DD}$ in  $\PP(\Lambda_o\otimes\CC)$ defined by $\omega\cdot\omega =0$. The open subset $\DD$ of $\check{\DD}$ defined $\omega\cdot\bar\omega <0$ has two  connected components that are interchanged by complex conjugation as well as by an element of $\G$. We put $X:=\G\bs \DD$.
The \emph{basic automorphic line bundle} $\Acal(1)$ on $\DD$ is the restriction of 
$\Ocal^\an_{\PP(\Lambda_o\otimes \CC)} (-1)$ to $\DD$. It is acted on by $\G$ and hence descends to a
line bundle over $X$, denoted $\Ocal_X(1)$, in the sense of orbifolds. (The notational switch from
$-1$ to $1$ has to with the fact that this bundle turns out to be  ample.)
A section of $\Ocal_X(k)$ is by definition a $\G$-invariant section of $\Acal (k)$. 
The Baily-Borel theory 
tells us among other things that 
\[
\oplus_{k\ge 0} H^0(X,\Ocal_X(k))= \oplus_{k\ge 0} H^0(\DD,\Acal (k))^\Gamma
\]
is a finitely generated graded algebra (of automorphic forms) whose $\proj$ defines a normal projective completion $X\subset X^\bb$  of the orbit space. Its boundary $X^\bb -X$ is of dimension at most one
and naturally stratified: we add a singleton resp.\ an irreducible  curve for every $\G$-orbit of primitive isotropic sublattices of rank $1$ resp.\ $2$ (with the incidence relations faithfully reflecting the inclusion relations). 
So in the present case we have by Lemma \ref{lemma:iii} and Corollary \ref{cor:ii} the following strata:
a singleton $X(III)$ and irreducible curves $X(R)$, with $R$ running over the root systems
$2E_8$, $D_{16}$, $A_{17}$, $E_7\perp D_{10}$, $3E_6$, $D_7\perp A_{11}$. The curves have the 
singleton $X(III)$ as common boundary. 

We also use  the set $\Hcal$ of special vectors to index the collection of hyperplanes in $\Lambda_o\otimes\CC$ or  $\PP(\Lambda_o\otimes\CC)$ that are orthogonal  to such vectors; in particular we denote by 
$\DD_h$ the hyperplane section of $\DD$ defined by $h\in\Hcal$. Thus we get a $\G$-invariant arithmetic arrangement on $\DD$ in the sense of  \cite{looijenga}. We denote the image of any $\DD_h$ in $X$ by
$X_\Hcal$. Since $X$ is an orbifold, $X_\Hcal$ is a Cartier divisor in $X$  the orbifold sense.
Its closure $\overline{X}_\Hcal$ in $X^\bb$  contains a given boundary stratum if and only if there exists a 
special vector perpendicular to a primitive isotropic sublattice representing that stratum. This closure is disjoint with the remaining strata. Since a special vector is a multiple of a short root, we can immediately tell when this is the case:

\begin{lemma}
The closure  $\overline{X}_\Hcal$ of $X_\Hcal$ contains the one dimensional strata of type $X(R)$ 
for which $R$ is a root system of rank $<18$ (so $R$ of type
$2E_8$, $D_{16}$, $A_{17}$, $E_7\perp D_{10}$) and  the punctual stratum $X(III)$, but is disjoint with
the others (the strata $X(R)$ with $R$ of type $3E_6$ and $D_7\perp A_{11}$).
\end{lemma}

Notice that a  subset of $\Hcal$ spans a positive definite sublattice if and only if its orthogonal complement meets $\DD$.

\begin{corollary}[to Lemma \ref{lemma:specialpos}]
The only proper intersections of the arrangement  $\{\DD_h\}_{h\in\Hcal}$  are of codimension $2$ and of type $G_2$. These lie in a single $\G$-conjugacy class.
\end{corollary}

At this point we need to recall some the results of \cite{looijenga}, but
we do that in manner that we hope is easiest on the reader.

The closure 
$\overline{X_\Hcal}$ of  $X_\Hcal$ in $X^\bb$ is not a $\QQ$-Cartier divisor.
According to Proposition 7.2 of  \cite{looijenga} the normalized blowup of $\overline{X_\Hcal}$  in $X^\bb$ that we denote here by $\widetilde{X^\bb}\to X^\bb$ has the property that the preimage $\widetilde{X(R)}\to X(R)$ of $X(R)$ is proper and  flat with fiber dimension $18-\rk(R)$. Such a stratum is in fact constructed in terms of a  sublattice of $\Lambda_o$ spanned by 
a primitive isotropic sublattice of rank two and the special vectors perpendicular to it.
The preimage $\widetilde{X(III)}\to X(III)$  is of dimension two and is constructed in terms of a semipositive sublattice of $\Lambda_o$ of rank $3$ spanned by a primitive isotropic sublattice of rank one and two special vectors perpendicular to it. 

An important feature of this construction is the following: The codimension $2$  intersections $\DD_h\cap \DD_{h'}$ define in $X_\Hcal$ a hypersurface (that we shall denote by $X'_\Hcal$) with the property that if we also blow up the strict transform of $\overline{X'_\Hcal}$ in $\widetilde{X^\bb}$, then the divisors over $X_\Hcal$ and $X'_\Hcal$ can be contracted in the ambient variety (in the opposite direction, like flops) onto a curve resp.\ a singleton.
We regard this contracted variety as a compactification of $\Xint:=X-X_\Hcal$. As such it is very much like the Baily-Borel compactification (that is why we shall denote it by
$\Xint^\bb$) since it  may be characterized by the fact that 
\[
\Xint^\bb=\proj \oplus_{k\ge 0} H^0(\Xint ,\Ocal (k))= \oplus_{k\ge 0} H^0(\Dint,\Acal (k))^\Gamma,
\]
where  $\Dint:=\DD-\cup_{h\in\Hcal}\DD(h)$ (so that $\Xint:=\G\bs \Dint$). 
The  boundary of $\Xint$ in $\Xint^\bb$ now comes with a decomposition into orbifolds indexed as below. 
\begin{itemize}
\item[$\Xint(I_0)$] a singleton (the contraction of the divisor over
$\overline{X_\Hcal}$),
\item[$\Xint(I_1)$] a curve (the contraction of the  divisor over 
$\overline{X'_\Hcal}$),
\item[$\Xint(II_1)$]  curves $\Xint(3E_6)$ and $\Xint(D_7\perp A_{11})$, 
\item[$\Xint(II_2)$]  surfaces  $\Xint(A_{17})$, $\Xint(E_7\perp D_{10})$, 
\item[$\Xint(II_3)$]  threefolds $\Xint(2E_8)$, $\Xint(D_{16})$,
\item[$\Xint(III_0)$] a singleton, 
\item[$\Xint(III_1)$] a curve,
\item[$\Xint(III_2)$] a surface.
\end{itemize}
This is a stratification in the sense that the closure of a member is a union
of members.  The incidence scheme is dictated by lattice embeddings:
\[
\begin{array}{ccccc}
&&I_0&<& I_1\\
&&\wedge&&\wedge\\
III_0&<&III_1&<&III_2\\
\wedge&&\wedge&&\wedge\\
\{II_1\}&&\{II_2\}& &\{II_3\}
\end{array}
\]
The minimal strata are the two singletons represented  by
$I_0$ and $III_0$. The maximal strata (whose closures yield 
the irreducible components of  the boundary) are those represented at the bottom and on the right:
three curves: $I_1=I(2E_8\perp 2U)$ and the two curves that make up $II_1$, $II(3E_6)$ and  $II(D_7\perp A_{11})$,  four surfaces: $III_2=III(2E_8\perp U)$ and the two surfaces that make up $II_2$:  $II(A_{17})$ and $II(E_7\perp D_{10})$, and two threefolds: $II(E_8\perp E_8)$ and $II(D_{16})$ (which make up $II_3$).

\begin{remark}
The Baily-Borel compactification $X^\bb$ arises as the $\G$-orbit space of a natural $\G$-equivariant extension $\DD^\bb\supset \DD$. The boundary  $\DD^\bb- \DD$ 
is naturally and $\G$-invariantly decomposed into  strata (in this case consisting of copies of the upper half plane and singletons) so that this stratification descends the one of the boundary $X^\bb-X$.
Something similar is the case for the compactification $\Xint^\bb$:
it is obtained as the $\G$-orbit space of a natural $\G$-equivariant extension 
$\Dint^\bb\supset \Dint$ whose boundary is naturally and $\G$-invariantly stratified
that descends to a stratification of $\Xint^\bb-\Xint$.  

If $S\subset  \Dint^\bb$ is a stratum, then the group $Z_\G(S)$ of $\gamma\in\G$ that
leave $S$ pointwise fixed is relevant for understanding the transversal structure of the
image of $S$ in $\Xint^\bb$: the $Z_\G(S)$-orbit space of the star of $S$ (the union of
strata havine $S$ in thier closure) is in a natural way a normal analytic space  and
the natural map from that orbit space to $\Xint^\bb$ is a local isomorphism along $S$.
 In the algebro-geometric context, the group $Z_\G(S)$
has an interpretation as a local monodromy group. For instance, if $S$ is a singleton that lies over the singleton $\Xint (III)$, then $Z_\G(S)$ is isomorphic  the semidirect product of the Weyl group with Dynkin diagram $B_\ell$ (that appears in  \ref{properties})
and its root lattice $\Lambda_1$. The main theorem of this paper  implies
that this is the local  monodromy group of the cubic fourfold defined by $Z_0Z_1Z_2=Z_3Z_4Z_5$  in $\PP^5$.
\end{remark}

\section{The period map}\label{sect:periodmap}
We fix a $6$-dimensional complex vector space $V$ and a generator  $\mu\in\wedge^6V^*$. For some of what follows we also need a hermitian inner product on $V$ and although this serves only an auxiliary purpose, we fix that as well. 
\\

Let $Y\subset \PP(V)$ be a cubic hypersurface (regarded as a divisor) and let  $F\in \CC[V]_3\cong \sym^3(V^*)$ be an equation for $Y$. We regard $\mu$ as a translation-invariant $6$-form on $V$ so that $F^{-2}\mu$ is 
rational $6$-form that is invariant under scalar multiplication. The residue of
this form at the hyperplane at infinity  is a rational $5$-form $\tilde\omega_F$ on $\PP (V)$ with a second order pole along $Y$. We can take the residue once more on the smooth part 
$Y_\reg$ of $Y$ in the sense of Griffiths to produce a class $[\omega_F]\in H^4(Y_\reg,\CC)$: it is characterized by the fact that the value of $[\omega_F]$  on a $4$-cycle in $Y_\reg$ is the integral of $\tilde\omega_F$ over the pre-image of that $4$-cycle in a tubular neighborhood boundary of $Y_\reg$ in $\PP(V)$. We can do this naturally on the form level (so that a $4$-form $\omega_F$ on $Y_\reg$ is defined) with the help of a hermitian inner product in $V$ (which yields a Fubini-Study metric on $\PP(V)$), see \cite{griffiths}. This form has Hodge level $3$ in the sense that it is a linear combination of a form of type $(3,1)$ and one of type $(4,0)$. It is clear that the dependence of $\omega_F$ 
on $F$ is homogeneous of degree $-1$. 

Suppose now that $Y\subset\PP(V)$ is nonsingular. Then $H^4(Y)$ is a unimodular odd lattice of signature $(21,1)$. If $y\in H^2(Y)$ is the hyperplane class, then
$y^2\in H^4(Y)$ has selfintersection $y^4=3$. The classical Lefschetz theory affirms that
the orthogonal complement of $y^2$ in $H^4(Y)$ is generated by vanishing cycles and these have selfintersection $2$ this orthogonal complement is even.
So there exists an isometry $\phi:H^4(Y)\to \Lambda$ that sends
$y^2$ to $\eta$. Such an isometry is called a \emph{marking}. It is clear that these
markings are simply transitively permuted by $\G$. It is well-known that the nonzero Hodge numbers of $Y$ in degree $4$ are $h^{3,1}(Y)=h^{1,3}(Y)=1$ and $h^{2,2}(Y)=21$. According to the Griffiths theory, $H^{3,1}(Y)$ is spanned by $[\omega_F]$ and so $[\omega_F]\cdot [\omega_F]=0$ and
$[\omega_F]\cdot\overline{[\omega_F]}<0$. We also have that $[\omega]\cdot y^2=0$.
So the marking associates to $Y$ an element of $\DD$  (its period point). Moreover,  
the line in $\Lambda_o\otimes\CC$ defined by that point is identified with the dual of the 
second  tensor power of the line of equations for $Y$ (the possible $F's$). If we forget about the marking,
then $Y$ defines an element of $X=\Gamma\bs \DD$ and the line of equations for $F$ raised to the tensor power $-2$ gets identified with the `automorphic line' over that element. This identification is
canonical in the sense that it is constant on the $\GL(V)$-orbit of $F$ in $\sym^3V^*$.

We better do this universally.  Let us abbreviate the $\GL (V)$-representation  $\sym^3V^*$  by $T$ and
let $\Ycal\subset \PP(V)_T$ be the universal cubic. The latter is given by a single equation 
$F\in \CC[T\times V]$. We denote by $T^\circ$ the locus where $\Ycal$ is smooth over $T$ so that 
$\Ycal_{T^\circ}\subset\Ycal^\circ$. Denote by $\GL (V)$-orbit space of $T^\circ$ by $\Mcal^\circ$.
This is the moduli space of smooth cubic $4$-folds and the previous discussion produces a morphism
$P: \Mcal^\circ\to X$ covered by an identification of $P^*\Acal(1)$ with $\Ocal_{\Mcal^\circ}(2)$, where
$\Ocal_{\Mcal^\circ}(k)$ stands for the orbifold line bundle over $\Mcal^\circ$ that comes from
$\Ocal_{\PP(T)}(k)$. 

This discussion essentially subsists if  the singular points of $Y$ are all simple, that is,
of type $A$, $D$ or $E$: then any one parameter smoothing of $Y$ has finite monodromy so that a 
finite base change eliminates the monodromy altogether. By a theorem of Griffiths the period map then extends to the whole base. So if we denote the corresponding open subset of $T$ by $\Tint$ 
and denote its $\GL (V)$-orbit space by $\Mcalint$, then we have a period map
\[
P: \Mcalint\to X.
\]
It is well-known (and not that difficult to show) that $P$ is a local isomorphism. 
Much harder is the theorem of Voisin \cite{voisin} that asserts that $P$ is injective.
We do not want to make use that theorem, but rather reprove it along the way.
Our main result may be stated as follows.

\begin{theorem}\label{thm:main}
The period map for cubic fourfolds with at most  simple singularities, $P: \Mcalint\to X$,
 is an open embedding with image $\Xint$. It identifies the automorphic line bundle restricted to $\Xint$
 with the line bundle $\Ocal_{\Mcal^\circ}(2)$ so that we obtain an isomorphism of  $\CC$-algebras
 \[
 \oplus_k H^0(\Dint , \Acal (k))^\Gamma\to \CC[\sym^3V^*]^{\SL (V)}
 \]
 which multiplies the degree by $2$. The passage to $\proj$, makes the above embedding extend to 
 an isomorphism of the GIT completion of $\Mcalint$ onto  the Baily-Borel type compactification 
 $\Xint^\bb$ of $\Xint$.
\end{theorem}

We should perhaps point out that since $-1\in \SL (V)$ acts as $-1$ on $\sym^3V^*$, 
the $\SL(V)$-invariants on  $\sym^3V^*$ have even degree.

We recall that the GIT completion $\Mcal$ of $\Mcalint$ is $\proj(\CC[\sym^3V]^{\SL (V)})$ (and so implicit in
this theorem is the statement that cubic fourfolds with singularities of type at most  $A$, $D$ or $E$ are stable).
The geometric invariant theory for cubic fourfolds has been worked out by Yokoyama \cite{yokoyama}
(see also Allcock \cite{allcock}) and more fully by Laza \cite{laza}.  We need the following.

\begin{theorem}[Yokoyama, Laza]
Every cubic fourfold with at most simple singularities is stable. 
The GIT boundary $\Mcal-\Mcalint$ is stratified with each stratum parameterizing 
fourfolds of the same topological type near their nonsimple singular locus. 
There are two minimal strata: the singleton $\Mcalint(I_0)$
represented by the orbit of the secant variety of the Veronese variety  and the singleton $\Mcalint(III_0)$
representing the fourfold admitting the equation $S_oS_1S_2=T_0T_1T_2$. 
Only two strata are not incident with $\Mcalint(I_0)$: the stratum $\Mcalint(3E_6)$ parameterizing fourfolds with three simple singularities of type $\tilde E_6$  and another that we denote by $\Mcalint(D_7\perp A_{11})$. Both are of dimension one, are open in $\Mcal-\Mcalint$ and have $\Mcalint(III_0)$ as  boundary.
\end{theorem}

\section{Boundary pairs}

We shall use a technique introduced in \cite{ls} that we presently recall. 
If $\Ycal\subset \PP(V)\times\Delta$ is a smoothing in $\PP(V)$ of a cubic fourfold  $Y$ (so $Y=Y_0$ and $Y_t$ is smooth for $t\not=0$), then an equation for $\Ycal$ (whose coefficients are holomorphic functions on $\Delta$) leads via the construction in Section \ref{sect:periodmap} to a relative $4$-form $\{\omega (t)\}$ on $(\Ycal -Y_\sing)/\Delta$. Any class $u\in H_4(Y_\reg)$  can be displaced to nearby  fibers to produce a flat family $\{ u(t)\in H_4(Y_t)\}_{t\in\Delta}$ such that $\int_{u(t)}\omega(t)$ is continuous (in fact holomorphic) on $\Delta$. The following Lemma is extracted (and its simple proof reproduced) from \cite{ls}.

\begin{lemma}\label{lemma:bp}
Let $Y\subset \PP(V)$  be a cubic fourfold (with equation $F$),  $\Ycal\subset \PP(V)\times\Delta$ a smoothing of $Y$ and $u\in H_4(Y_\reg)$ such that the following two \emph{Boundary Conditions} are fulfilled
\begin{enumerate}
\item[(B1)] $\int_{Y_\reg}\omega_F\wedge\overline{\omega_F}=-\infty$  and 
\item[(B2)] the deformation of $u$ to the generic fiber is not Poincar\'e dual to a multiple of the square of the hyperplane class.
\end{enumerate}
Let $S\subset\Delta^\times$ be  a sector and let a marking over $S$ be given so that is defined a period map $P: S\to \DD$. Then $u$ (pushed to nearby  fibers) 
becomes a nonzero linear form on  $\Lambda_o$ whose kernel  defines a projective
hyperplane in $\PP(\Lambda_o\otimes\CC)$ with the property that it contains any accumulation point of $P(s)$, $s\in S\to 0$. 
\end{lemma}
\begin{proof}
We push $u$ to nearby  fibers to produce a flat family 
$\{ u(t)\in H_4(Y_t)\}_{t\in\Delta}$  and choose an equation for $\Ycal$ as above so that we have a relative $4$-form $\{\omega (t)\}$ on $(\Ycal -Y_\sing)/\Delta$ for which $\int_{u(t)}\omega(t)$ is continuous. 
For $t\not= 0$, we have that
\[
[\omega (t)]\cdot \overline{[\omega (t)]}=\int_{Y_t} 
\omega (t)\wedge \overline{\omega (t)}
\]
The latter tends to $\int_{Y_\reg}\omega (t)\wedge \overline{\omega (t)}=-\infty$  as $t\to 0$. This means that there exists a horizontal family $\{v(t)\in H_4(Y_t)\}_{t\in S}$ such that
$|\int_{v(t)} \omega (t)|\to\infty$  as $t\to 0$ and hence that
\[
\lim_{t\to 0} \frac{\int_{u(t)}\omega(t)}{\int_{v(t)}\omega(t)}=0.
\]
This is a property that only involves the behavior $\{H^{3,1}(Y_t)\}_{t\in S}$ and yields
a linear constraint. Property (B2) ensures that this constraint is nontrivial: $u$ defines a nonzero linear form on the primitive cohomology of a smooth fiber.
\end{proof} 

We can  exploit the openness of (B2) 
in conjunction with (B1) to sharpen Lemma \ref{lemma:bp} a little as follows:

\begin{lemma}\label{lemma:bp2}
Suppose $Y\subset \PP(V)$ is a cubic $4$-fold  that satisfies (B1) and admits
a one-parameter deformation $\Ycal\subset\PP(V)\times\Delta$ such that every
fiber $Y_t\not= Y$ satisfies (B2). Then the conclusion of Lemma \ref{lemma:bp} holds
for all fibers $Y_t\not=Y$ of $\Ycal/\Delta$ close to $Y$ with $u$ imposing the same linear constraint on the period map for each of them. (If $Y$ also satisfies (B2), then
$u$ imposes the same linear constraint for all fibers of $\Ycal/\Delta$ near $Y$.)
\end{lemma}
\begin{proof} This is easy.
The class $u\in H_4(Y_\reg)$ displaces to a class $u(t)\in H_4(Y_{t,\reg})$ for $t$ close to $0$.  So if $t\not=0$, then Lemma \ref{lemma:bp} applies and all the assertions follow.
\end{proof}

Condition (B1) can be verified as follows.

\begin{lemma}\label{lemma:b1}  
Condition \ref{lemma:bp}-(B1) is satisfied if  every singular point of $Y$ that is not of type $A$, $D$ or $E$ admits a local-analytic equation that is weighted homogeneous (with nonnegative weights, not all zero) such that twice  the degree of the equation is at least the sum of the weights. 
\end{lemma} 
\begin{proof}
Let $f\in \CC[z_1,\dots ,z_5]$ define such a singularity with the varibale ordered such that the weight $w_i$ of $z_i$ is zero for $i<k$ and positive fo $i\ge k$. 
Then $f^{-2} dz_1\wedge \cdots \wedge  dz_5$ defines a relative 
$4$-form  $\omega$ on the smooth part $U$ of $f=0$. This form has a negative weight, say $-d$.
The form $-\omega\wedge \overline{\omega}$ is positive (relative to the complex orientation)
wherever it is defined and $\CC^\times$ acts on it via the absolute value map with 
weight $-2d$. 
Let for $0\le a<b$, $r>0$, $U_r[a,b]$ be the part of $U$ where $|z_i|\le r$ for $i<k$ and 
$a\le \sum _{i\ge k} |z_i|^{1/w_i}\le b$. We have that  
\[
\int_{U_r[ta,t b]}\omega\wedge \overline{\omega}=t^{-2d} \int_{U_r[a,b]}\omega\wedge \overline{\omega}
\]
and so $\int_{U_r[a,b]}\omega\wedge \overline{\omega}= c(b^{-2d}-a^{-2d})$, where 
$c$ is a positive constant.
This proves the lemma for this particular equation with particular $\omega$. 
The general case is obtained when this $\omega$ is multiplied by a unit $u(z)=u_0+\sum u_iz_i+\cdots$ with $c\not=0$. Then 
\[
\int_{U_r[ta,t b]} |u|^2\omega\wedge \overline{\omega}= 
c(b^{-2d}-a^{-2d})|u_0|^2t^{-2d}   +\sum_{i\ge k}\tilde u_i t^{-2d+w_i}+\cdots
\]
in which  the first term is dominant and so $\int_{U_r[0,b]}u\omega\wedge \overline{u\omega}=-\infty$.
\end{proof}

The preceding lemma applies for instance to the simple elliptic singularities of type
$\tilde E_6$, $\tilde E_7$ and $\tilde E_8$ and to a one dimensional singular locus of transversal type $A_1$ (with local equation $z_1^2+z_2^2+z_3^2+z_4^2$).

On the other hand, Condition (B2) looks harder to establish as it involves all deformations of $Y$. It  is implied however by each of the following two properties that only regard 
$Y$:

\begin{lemma}  
Condition \ref{lemma:bp}-(B2) is satisfied if one the following holds:
\begin{gather*}
\int_u \omega_F\not= 0\quad  \text{ or } \tag*{($B2'$)} \\
\begin{vmatrix}
u\cdot u &\la y^2,u\ra \\
\la y^2,u\ra & y^2\cdot y^2
\end{vmatrix}
=3(u\cdot u)-\la y^2,u\ra^2\not=0. \tag*{($B2''$)}
\end{gather*}
\end{lemma} 
\begin{proof}
In the first case, use the fact that for an equation for a deformation 
$\Ycal\subset \PP(V)\times\Delta$ of $Y$ (in $\Ocal_\Delta\otimes \sym^3V^*$) determines an extension of $\omega_F$ to a relative form on the part where $\Ycal$ is smooth over $\Delta$ such that its integral over the deformed $u$ is continuous on $\Delta$ and hence nonzero near $o\in\Delta$. In the second case  the expression in question  is precisely the self-intersection of the primitive part of $u$. If that is nonzero, then so is the primitive part of the deformed $u$.
\end{proof}

In the situation of Lemma \ref{lemma:bp} we distinguish cases  as follows.  

If $3(u\cdot u)-\la y^2,u\ra^2<0$, then $H\cap \Lambda_o$ has signature $(20,1)$ and hence  $\PP (H\otimes\CC)$ cannot meet  the closure if $\DD$. So this case will  not occur. 

(I) If $3(u\cdot u)-\la y^2,u\ra^2>0$, then $H\cap \Lambda_o$ has signature $(19,2)$,  $\PP (H\otimes\CC)$ meets $\DD$ in its interior and the resulting hyperplane section is two copies
of the symmetric domain of the orthogonal group of $H\otimes\RR$.

(II,III) Suppose now  $3(u\cdot u)-\la y^2,u\ra^2=0$ and $\int_u \omega_F\not= 0$.
The last condition ensures that $\tilde u\not=0$ so that $H\cap \Lambda_o$ is of corank
$1$ in $\Lambda_o$  and the first condition then says
that $H\cap \Lambda_o$ is degenerate with rank one dimensional nillattice.
If we denote the latter by $K$, then $\PP (H\otimes \CC)$
meets the boundary of $\DD$ in a set that can be identified with the set of rays in $(\Lambda_o/K)\otimes\RR$  on which the form is $\le 0$ (this consists of the closures
of two closed real hyperbolic disks of dimension $19$). 

(II) This is a subcase of the previous case and refers to the situation when there exist  $u_1,u_2\in H_4(Y_\reg)$ for which the
the intersection matrix on $y^2,u_1,u_2$ has rank one and for which 
$\int_{u_1}\omega_F$ and $\int_{u_2}\omega_F$ are $\RR$-independent. The same argument proves that $u_1$ and $u_2$ determine a primitive  isotropic lattice $K\subset\Lambda_o$ of rank two such that any limiting value of $P(t)$ lies in a
codimension two linear subspace $\PP((K^\perp\cap\Lambda_o)\otimes\CC)$. The latter meets the closure of $\DD$ in the two half spheres that make up 
$\PP (K\otimes\CC)-\PP (K\otimes\RR)$. These are both one-dimensional  boundary components of  $\DD$ (each in a different component of $\DD$).

\section{The Boundary Conditions for some semistable fourfolds}
In this section we verify the Boundary Conditions \ref{lemma:bp} in a number of 
instances. Inspection of the lists shows:

\begin{lemma}\label{lemma:b1list}
Any minimal strictly semistable cubic threefold satisfies the conditions of Lemma \ref{lemma:b1}.
\end{lemma}

We now  turn to the two most important cases.

\subsection*{Type $I_0$ : the secant variety of a Veronese}
Fix a vector space $W$ of dimension  3 and consider the space 
$\sym^2W$. It has dimension 6 and parameterizes the quadratic forms on $W^*$.
The Veronese variety is the projectivization of the cone of
the elements of the pure squares in $\sym^2W$, and so the image of
$\PP(W)$ in $\PP(\sym^2W)$. The secant variety $Y$ of this variety consists of the elements that can be written as the sum of two squares, in other words is
the locus of  singular  quadratic expressions. Its singular part is the Veronese variety
(the locus of pure squares) and hence the smooth part $Y_\reg\subset Y$ is the locus
of  the quadratic expressions of exact rang $2$ (both are $\SL (W)$-orbits). 

\begin{lemma}
The  natural map $H_4(Y_\reg)\to H^4(Y)$ has a kernel  cyclic of order $2$ and image an infinite cyclic group  generated by a  distinguished element $a$. A lift $u\in H_4(Y_\reg)$ of $a$ satisfies $u\cdot u=3$. If $y\in H^2(Y)$ denotes the hyperplane class and if we identify
$H^8(Y)$ with $\ZZ$ by means of integration over the fundamental class, then $y^4=3$,
$a\cdot y^2=1$ and $3a-y^2=2h$ for some $h\in H^4(Y)$ (so that $h\cdot y^2=0$, $h^2=6$ and $2h+y^2$ is divisible by $3$).
\end{lemma}
\begin{proof}
The map that assigns to $q\in\sym^2W$ the minimal subspace $W_q\subset W$ 
such that $q\in\sym^2W_q$  gives us a fiber bundle 
$Y\to\check{\PP}(W)$:  The fiber over $[W']\in  \check{\PP}(W)$ is the
space of nonsingular elements  in $\PP(\sym^2 W')$, that is the complement of the 
conic $C(W')$  that is the image of $\PP(W')\to \PP(\sym^2 W')$.
Denote by
\[
\pi: \tilde Y\to \check{\PP}(W)
\] 
the projective plane bundle over whose fiber over 
$[W']\in \check{\PP}(W)$ is $\PP(\sym^2 W')$ and  denote by $C\subset \tilde Y$ the
corresponding family of conics. The evident map $\tilde Y\to Y$ is a resolution of singularities with $C$ as exceptional divisor. The contraction of $C$ to  the Veronese variety defines a $\PP^1$-bundle $C\to \check{\PP}(W)$. 

If $\zeta$ denotes the tautological plane bundle over $\check{\PP}(W)$, then
$C\to \check{\PP}(W)$ resp.\  $\tilde Y$ is the projectivization of the vector bundle 
$\zeta$ resp.\ $\sym^2\zeta$ over  $\check{\PP}(W)$. 
Denote by $u\in H^2(\check{\PP}(W),\ZZ)$ the positive generator, that
is the first Chern class of the line bundle $\Ocal_{\check{\PP}(W)}(1)$.
We have an exact sequences 
\begin{align*}
0\to \zeta\to &\Ocal_{\check{\PP}(W)}\otimes W\to \Ocal_{\check{\PP}(W)}(1)\to 0,\\
0\to \sym^2\zeta\to &\Ocal_{\check{\PP}(W)}\otimes\sym^2W\to 
\Ocal_{\check{\PP}(W)}(1)\otimes W\to 0,
\end{align*}
and hence 
\begin{align*}
c(\zeta)&=(1+u)^{-1}=1-u+u^2,\\
c(\sym^2\zeta)&=(1+u)^{-3}= 1-3u+6u^2. 
\end{align*}
If $x\in H^2(\tilde Y)$ resp.\ $y\in H^2(\tilde Y)$ denotes the first Chern class of 
$\Ocal_{C}(1)$ resp.\ $\Ocal_{\tilde Y}(1)$, then  by a formula of Grothendieck we have that
\begin{align*}
H^\pt(C)&=H^\pt( \check{\PP}(W))[x]/(x^2-ux+u^2),\\
H^\pt(\tilde Y)&=H^\pt( \check{\PP}(W))[y]/(y^3-3uy^2+6u^2y).
\end{align*}
The restriction map $H^\pt(\tilde Y)\to H^\pt(C)$ is $H^\pt( \check{\PP}(W))$-linear,
but as the restriction of $\Ocal_{\tilde Y}(1)$ to $C$ is $\Ocal_{C}(2)$ it will
send $y$ to $2x$ (notice that it is indeed true that  
$(2x)^3-3u(2x)^2+6u^2(2x)=-12u(x^2-ux+u^2)$).
If we feed this into the exact sequence
\begin{multline*}
\cdots\to H^2(\tilde Y)\to H^2(C)\to H^4_c(Y_\reg)\to H^4(\tilde Y)\to H^4(C)\to\cdots  ,
\end{multline*}
then we see that the middle term  is the direct sum of a copy of $\ZZ/2$ and an infinite cyclic group generated by a  class $u$ that maps to $y^2-2uy+4u^2$. So
\[
u^2=(y^2-2uy+4u^2)^2=y^4-4uy^3+12 u^2y^2
\]
(we used that $u^3=0$). In order to compute this, we observe that
$H^8(\tilde Y)$ is generated by $y^4, uy^3, u^2y^2$ and that we have 
$uy^3=u(3uy^2-6u^2y)=3u^2y^2$ and $y^4=y(3uy^2-6u^2y)=3uy^3-6u^2y^2=3u^2y^2$.
So $u^2y^2$ is the orientation class and so  
\[
a^2= 3u^2y^2-4.3u^2y^2+12 u^2y^2=3u^2y^2.
\]
So if we regard $u$ as an element of $H_4(Y_\reg)$, then we see that its self-intersection is $3$. The pull back of $\Ocal_{\PP(\sym^2W)}(1)$ along the map 
map $\tilde Y\to Y\subset \PP(\sym^2W)$  is $\Ocal_{\tilde Y}(1)$. So 
the hyperplane class of $\PP(\sym^2W)$ is mapped to $y$ and
$a\cdot y^2=(y^2-2uy+4u^2)\cdot y^2=3-6+4=1$
(and from $y^4=3$ we see that $Y$ is indeed of degree $3$). 
Now consider the projection map $\tilde Y\to Y$. It contracts $C$ along the projection
$C\to \check{\PP}(W)$ (a $\PP^1$-bundle). This implies that we have a long exact sequence
\[
\dots\to H^1(\check{\PP}(W))\to H^4(Y)\to H^4(\tilde Y)\to H^2(\check{\PP}(W))\cdots
\]
Since $H^1(\check{\PP}(W))=0$ and $H^2(\check{\PP}(W))$ is torsion free, the  map 
$H^4(Y)\to H^4(\tilde Y)$ must be a primitive embedding. Therefore, for the assertions
to be proved, we may replace $Y$ by $\tilde Y$. It is clear that the element
$h:=y^2-3uy+6y^2$ which evidently satisfies $2h=3(y^2-2uy+4u^2)-y^2$ possesses also the other desired properties. 
\end{proof}

This lemma implies that we are dealing with a  type (I) situation for which the hyperplane is special:

\begin{corollary}\label{cor:special}
This $Y$ and $u$ satisy the hypotheses of Lemma \ref{lemma:bp} with  the projective hyperplane
perpendicular to a special vector.
\end{corollary}
\begin{proof}
A calculation near the double point locus of $Y$ shows
that $\omega_F$ is not square integrable and we have $3(u\cdot u)-\la y^2,u\ra^2=3.3-1=8$, so that we are in case $I$.
The smoothing  and the marking determine $\tilde u\in\Lambda$ and $\tilde h\in\Lambda_o$ satisfying $\tilde h\cdot \tilde h=6$ and 
$\eta -\tilde h=3(\tilde u-\tilde h)$.  Hence  $\tilde h$ is a special vector in $\Lambda_o$ and
the limiting values of the period on $S$ takes values in the special hyperplane
in $\PP(\Lambda_o\otimes\CC)$ perpendicular to it.
\end{proof}

\subsection*{Type $III_0$}
Write $V$ as  a direct sum of  subspaces of dimension 3 and denote its coordinates accordingly:
$(S_0,S_1,S_2,T_0,T_1,T_2)$. We consider the fourfold  $Y\subset \PP^5$ 
defined by the equation $F(S,T):=S_0S_1S_2-T_0T_1T_2$.
This is in fact a toric variety with the torus in question acting diagonally. The singular locus
$Y_\sing$ of $Y$ is the union of lines connecting a coordinate vertex in the plane $\PP(V')$ with coordinate vertex in  the plane $\PP(V'')$. The  evident projection of $Y$ onto $\PP(V')\times\PP(V'')$, is a morphism away from the union coordinate lines in the two planes  $\PP(V')$ and $\PP(V'')$.
We use the standard coordinates $(s_1,s_2,t_0,t_1,t_2)$ on the affine open subset 
$U(S_0)\subset \PP(V)$ defined by $S_0\not=0$. So $Y\cap U(S_0)$ is given by putting 
$f(s,t)=s_1s_2-t_0t_1t_2$ equal to zero. Under the morphism
\[
A:\CC^3\to Y, \quad (u,v,w)\mapsto(u,vw,v,w,u)
\]
the preimage of $Y_\sing$ is given by the union of the planes $u=v=0$ and $u=w=0$.
So $A$ maps the $\Sigma\subset \CC^3$ defined by $|u|^2+|v|^2=2$, $|v|=|w|$ to $Y_\reg$.
Denote by $\Sigma_-$ resp.\ $\Sigma_+$ the locus where $|u|\le 1$  resp.\ $|u|\ge 1$.
Then $(u,v,w)\in \Sigma_-\mapsto (u, v/|v|, w/|w|)$ identifies $\Sigma_-$ with 
a closed $2$-disk times a $2$-torus and $(u,v,w)\in \Sigma_+\mapsto (u/|u|, v, w)$
identifies $\Sigma_+$ with a circle times the cone over a $2$-torus. This also shows that
$A(\Sigma_-)$ resp.\ $A(\Sigma_+)$is contained in $U(T_0T_1)$ resp.\ $U(S_0S_1)$. 

The common boundary
$\Sigma_0:=\Sigma_-\cap\Sigma_+$ is a $3$-torus. We orient $\Sigma_0$ by means of the
identification with the standard $3$-torus and let $\Sigma_{\pm}$ be oriented such that
they induce the given orientation. Then $[\Sigma]:=[\Sigma_+]-[\Sigma_-]$ is a $4$-cycle.
We denote by $a\in H_4(Y)$ the class of $A_*[\Sigma]$. 

\begin{lemma}
We have $a\cdot a=0$ and if $y\in H^2(Y)$ is the hyperplane class, then $y^2\cdot a=0$.
Moreover, the value of $[\omega_F]\in H^4(Y_\reg,\CC)$ on $a$ is nonzero.
\end{lemma}
\begin{proof}
The locus $\Sigma'\subset \CC^3$ defined by $|u|^2+|v|^2=1$, $|v|=|w|$ defines 
a cycle $[\Sigma']$ homologous to $[\Sigma]$ in the preimage of $Y_\reg$. Since
$A(\Sigma)$ and $A(\Sigma')$ are disjoint, it follows that $a\cdot a=0$. 
Since $A$ takes values in an affine piece, we also have that $a\cdot y^2=0$.

In order to prove the last assertion,  we identify, following Griffiths' receipe, the  cohomology class on $Y_\reg$ defined by the double residue of $F^{-2}dS_0\wedge\cdots \wedge dT_2$ as a Cech-cocycle. The corresponding rational $5$-form $\tilde\omega_F$ on $\PP(V)$ is on  standard affine open piece  $U(S_0)$ given by  
\[
\tilde\omega_F=\frac{ds_1\wedge ds_2\wedge dt_0\wedge dt_1\wedge dt_2}{ f(s,t)^2},
\]
We write this as the exterior derivative of a $4$-form in two ways: $\tilde\omega_F=d\omega_{S_0S_1}$, where
\[
\omega_{S_0S_1}:=\Big(\frac{ds_1\wedge dt_0\wedge dt_1\wedge dt_2}{s_1f(s,t)} \Big).
\]
This is a form defined on $U(S_0S_1F)$.  We also have $\tilde\omega_F=d\omega_{T_0T_1}$, where
\[
\omega_{T_0T_1}:=\Big(-\frac{ds_1\wedge ds_2\wedge dt_0\wedge dt_1}{t_0t_1f(s,t)} \Big).
\]
is a form defined on $U(S_0T_0T_1F)$. Notice that the difference
\[
\omega_{T_0T_1}-\omega_{S_0S_1}=\frac{df}{f} \wedge 
\frac{ds_1\wedge dt_0\wedge dt_1}{s_1t_0t_1}
\]
has as residue on $Y_\reg$ the restriction $(s_1t_0t_1)^{-1}ds_1\wedge dt_0\wedge dt_1$ to $Y_\reg$.
Let $B\to Y_\reg$ be a tubular neighborhood of $Y_\reg$ in $\PP^5$ so that
$\partial B\to  Y_\reg$ is a circle bundle with total space contained in $\PP^5-Y$.
An application of Stokes' theorem yields
\begin{align*}
[\omega_F](a)&=\int _{\partial B | A_*[\Sigma]} \tilde\omega_F\\
&=\int _{\partial B | A_*[\Sigma_+]} d\omega_{T_0T_1}-\int _{\partial B | A_*[\Sigma_-]} d\omega_{S_0S_1}\\
&=\int_{\partial B | A_*[\Sigma_0]} \left(\omega_{T_0T_1}-\omega_{S_0S_1}\right) \text{by Stokes' theorem}\\
&=\int_{\partial B | A_*[\Sigma_0]}\frac{df}{f} \wedge \frac{ds_1\wedge dt_0\wedge dt_1}{s_1t_0t_1}
\end{align*}
and the residue theorem shows that the latter integral equals
\[
\int_{A_*[\Sigma_0]}\frac{ds_1\wedge dt_0\wedge dt_1}{s_1t_0t_1}=
\int_{|u|1} \frac{du}{u}\int_{|v|=1} \frac{dv}{v}\int_{|w|=1} \frac{dw}{w}=(2\pi \sqrt{-1})^3.
\]
\end{proof}

\begin{corollary}
This  $Y$ is a   boundary case of type $II$ or $III$.
\end{corollary}

\section{Proof of the main theorem}

The period maps defines a morphism $P: \Mcalint\to X$. It defines rational maps
$G_P:\Mcal\dashrightarrow X^\bb$ and $\Gint_P: \Mcal\dashrightarrow \Xint^\bb$.

\begin{lemma}\label{lemma:stratasep}
The map $\Gint_P$ sends the boundary $\Mcal-\Mcalint$ to the boundary $\Xint^\bb-\Xint$. Moreover, the preimage of $\Xint(3E_6)\cup \Xint(D_7\perp A_{11})$ is $\Mcalint(3E_6)\cup\Mcalint(D_7\perp A_{11})$.
\end{lemma}
\begin{proof}
We established that the singleton strata $\Mcalint (I_0)$ and $\Mcalint(III_0)$ satisfy the Boundary Conditions 
of Lemma \ref{lemma:bp} with $\Mcalint(I_0)$ of type (I) and $\Mcalint(III_0)$ of type 
(II-III). Every stratum $S$ of $\Mcal-\Mcalint$ satisfies the nonsquare integrability condition  (B1) and has at least one of these singletons in its closure. By Lemma \ref{lemma:bp2} it then also satisfies (B2). In fact, that lemma tells us that 
if $S\ge \Mcalint(III_0)$, then $S$  is of type (II-III) and  will be mapped by $G_P$ to the boundary  $X^\bb-X$. And if  $S\ge \Mcalint(I_0)$, then the lemma implies (in combination with
Corollary \ref{cor:special}) that $G_P$ maps $S$ to $\overline{X_\Hcal}$. In particular,
$\Gint_P$ maps  $\Mcal-\Mcalint$ to $\Xint^\bb-\Xint$.  

Now $\Mcalint(3E_6)$ and $\Mcalint(D_7\perp A_{11})$ are the only strata of $\Mcal-\Mcalint$ that are not $\ge \Mcalint(I_0)$. Likewise $X(3E_6)$ and $X(D_7\perp A_{11})$ are the only strata
of $X^\bb-X$ not  meeting $\overline{X_\Hcal}$. These strata are irreducible. So
the preimage of the union of the latter two is the union of the former two.
\end{proof}

The following proposition will imply Voisin's injectivity of the period map.

\begin{proposition}[Torelli property near a boundary component]\label{prop:degone}
The map $G_P$ maps $\Mcalint(3E_6)$ to $\Xint(3E_6)$ and is a local isomorphism along $\Mcalint(3E_6)$.
\end{proposition}

We relegate the proof to section \ref{sect:torelli}.

\begin{proof}[Proof of Theorem \ref{thm:main}]
Lemma \ref{lemma:stratasep} implies that the $\Gint_P$-preimage of $\Xint(3E_6)$ is 
$\Mcalint(3E_6)$ or $\Mcalint(D_7\perp A_{11})$ (the strata in question are irreducible).
Proposition \ref{prop:degone} tells us that it has to be  $\Mcalint(3E_6)$ and that 
$\Gint_P$ is there a local isomorphism. This implies that $\Gint_P$ has degree one.
Since $\Gint_P$ maps boundary to boundary, it follows that it restricts to
a proper morphism $P:\Mcalint\to \Xint$ of degree one.  As this is a local isomorphism 
of degree one between integral varieties, this restriction must be an isomorphism. This isomorphism takes the  automorphic
line bundle $\Ocal_X(1)$ to $\Ocal_{\Mcalint} (2)$. So it induces an isomorphism
of the algebra $\oplus_k H^0(\Xint , \Ocal_X(k))=\oplus_k H^0(\Dint , \Acal(k))^\G$ onto the even part of 
$\oplus_k H^0(\Mcalint , \Ocal_{\Mcalint}(k))=\CC [\sym^3V^*]^{\SL (V)}$ (degrees are multiplied by two).
Since $-1\in \SL(V)$ acts in $\sym^3V$ as minus the identity, the odd part of the latter algebra is zero.
Passing to the projs of these algebra yields that $G_P$ is in fact the graph of an  isomorphism.
\end{proof}

\section{Proof of the Torelli  property \ref{prop:degone}}\label{sect:torelli}
We first prepare the setting. Decompose $V$ into two  threedimensional subspaces  $V=V'\oplus V''$ and choose coordinates $S_0$, $S_1$, $S_2$, $T_0$, $T_1$, $T_2$ accordingly. An element of $\Xint(3E_6)$ is represented by a cubic fourfold $Y\subset \PP(V)$ that  has an equation of the form $F (S,T)=S_0S_1S_2- \Phi (T_0,T_1,T_2)$, where $ \Phi$ defines a nonsingular cubic plane curve $C\subset \PP(V'')$. So
$Y\cap \PP(V'')=C$.

 It is clear that $Y\cap \PP(V')$  consists of three coordinate lines in $\PP(V')$.  The vertices  $p_0,p_1,p_2$ of this coordinate triangle ($p_i$ is defined by putting all coordinates but $S_i$ equal to  nonzero) are the singular points of $Y$.  The singularity at $p_i$ is exhibited on the affine piece defined by $S_i=0$; for instance, for $i=0$, we find the  equation  $s_1s_2=\Phi (t_0,t_1,t_2)$ and so it is a singular point of type $\tilde E_6$. Similarly for $p_1$ and $p_2$.

Choose a smoothing $\Ycal/\Delta$ of $Y$ and let $Y_t$ be a smooth fiber of it.

\begin{lemma}
The image $I$ of $H_4(Y_\reg)\to H_4(Y_t)$ is a primitive isotropic sublattice.
We have a canonical identification of  $I$  with $H_1(C)$ and the long roots in
$I^\perp /I$ span a lattice $Q$ which decomposes into three $E_6$ lattices $Q_i$, $i=0,1,2$,
so that the preimage  of $Q_i$ in $I^\perp$  is the vanishing (Milnor) lattice of $p_i$.
The index of $Q$ in $I^\perp /I$ is three.
\end{lemma}
\begin{proof}
The link $L$ of $p_0$ in $Y$ has $H_4(L)$ isotropic of rank $4$. We 
know from singularity theory that there is a canonical isomorphism 
$H_4(L)\cong H_1(C)$ and that the form $\omega_F$ maps $H_4(L)$ 
onto a lattice in $\CC$. This last fact implies that $H_4(L)$ embeds in $H_4(Y_t)$. We provisionally denote the image of this embedding by $I_0$. 

We first verify that $I_0$ is primitive. If $F_t\subset Y_t$ denotes the Milnor fiber of 
$p_0$ in $Y_t$, then it well-known that $H_4(L)\to H_4(F_t)$ is primitive (for
$L$ can be identified with a boundary of $F$ and $H_4(F_t,\partial F_t)\cong H^4(F_t)$ is known  to be torsion free). It therefore suffices to show that $ H_4(F_t)\to H_4(Y_t)$
is primitive. But $Y_t/F_t$ is homeomorphic with the fourfold $Y'$ that we obtain  
by smoothing the points $p_1,p_2$ of $Y$, while retaining $p_0$. So it suffices to show that $H_4(Y')$ is torsion free. This follows from the following general argument:
let $X\subset Y'$ be a generic hyperplane section (so $X$ is a smooth cubic forurfold)
and consider the exact sequence
\[
H_4(Y'-Z)\to H_4(Y')\to H_2(X)
\]
It is known that $Y'-Z$ has the homotopy type of a bouquet of $4$-spheres 
and so $H_4(Y'-Z)$ is torsion free. It is also known that  $H_2(X)$ is infinite cyclic (hence torsion free). It follows that
$H_4(Y')$ is torsion free. So $I_0$ is primitive. 

We find similarly for $p_i$ ($i=1,2$) a primitive isotropic rank two sublattice
$I_i$ in $H_4(Y_t)$. We have $I_i\perp I_0$  for obvious geometric reasons.  But $I_0^\perp/I_0$ is positive definite and hence $I_i=I_0$. We now write $I$ for $I_0$.

The image of $H_4(L)\cong H_4(\partial F_t)\to H_4(F_t)$ is the kernel of the intersection pairing on $ H_4(F_t)$ and the residual lattice (denoted $Q_i$) is of type $E_6$. This is also true at the other singular points, so that we have in fact obtained
an embedding of $Q:=Q_0\perp Q_1\perp Q_2$ in $I^\perp/I$. 
We found in Corollary \ref{cor:ii} that the equivalence class of a primitive isotropic rank two lattices $J$ in $\Lambda_o$ is characterized by  the system of (long) roots in
$J^\perp/J$ and that the root lattice $3E_6$ occurs in this manner. Since there is no root lattice of rank $\le 18$ that strictly contains $3E_6$, this identifies the 
equivalence class of $I$. 

It remains to see that $Q$ is of index $3$ in $I^\perp/I$. 
The discriminant of $E_6$ is $3$ and hence the one of $Q$ equal to $3^3$. Since the square of $[I^\perp/I:Q]$ must divide the discriminant of $Q$, either $Q=I^\perp/I$ or 
$[I^\perp/I:Q]=3$. Let us exclude the former. For this we go back to the decomposed set 
$\Bcal=\Bcal'\sqcup \Bcal''$ and the Dynkin diagram $D(B_\ell)$ of long roots it defines. 
We  think of $\Bcal$ as the set of vertices of $B_\ell$ of degree $3$. We have an injection $r: B_\ell\to \Lambda_1$,  where we recall that $\Lambda_1= e^\perp/\ZZ e$, for some primitive isotropic $e\in\Lambda_o$.  The full subdiagram on $B_\ell -\Bcal'$ has $3$ connected components, each of type $\hat E_6$. 
These components have a common nilvector $e'\in \Lambda_1$ and  we may identify $Q$ with the span of the roots indexed by $B_\ell -\Bcal'$ modulo $\ZZ e'$. If $u,v\in\Bcal''$ are distinct, then  it is clear that if we fix a $\hat E_6$-component, then taking the inner product with  $r_u-r_v$ defines a linear form that takes the value
$1$ on an extremal vertex, $-1$ on another extremal vertex and is zero on all 
other vertices. From this it readily follows that $r_u-r_v$ is perpendicular to $e'$, but not  contained in any $\hat E_6$-summand.
\end{proof}

\begin{remark}\label{rem:discr}
The last part of this proof showed that we have a canonical identification of
the discriminant groups of the three $E_6$-summands: the subgroup
$(I^\perp/I)/Q$ lies in $Q^*/Q=Q_0^*/Q_0\oplus Q_1^*/Q_1\oplus Q_2^*/Q_2$
as a main diagonal, and thus induces natural isomorphisms between these summands.
This helps us to describe the quotient of the  orthogonal group  $\orth(I^\perp/I)$  by the
Weyl group $W_Q$ of $Q$ (which is indeed normal in $\orth(I^\perp/I)$):
as is well-known the orthogonal group of $Q_i$ is $\{\pm 1\}$ times its Weyl group.
Since $-1$ acts as such on  $Q_i^*/Q_i$, not every element of $\{\pm 1\}^3$ appears here: only its main diagonal (which acts as minus the identity in $Q$) preserves $I^\perp/I$.
 
It now easily follows that $\orth(I^\perp/I)/W_Q$  can be indentified with the product of 
$\{\pm 1\}$ and the permutation group ($\cong\Scal_3$) of the three summands. 
\end{remark}

We shall now assume that $C$ is generic in the sense that it has no exceptional
automorphisms.

\begin{lemma}
The stabilizer $\PSL(V)_Y$ of $Y$ in $\PSL (V)$ preserves $\PP(V')$ and $\PP(V'')$.
It is an extension of the stabilizer of $C$ in $\PSL(V'')$ by the stabilizer
of $S_0S_1S_2$ in $\SL (V')$. The former is the semidirect
product of an involution of $C$ in a flex point and the group
of order $3$ translations in $C$, $H_1(C,\mu_3)$; the latter 
is a semidirect product of the symmetric group on $S_0,S_1,S_2$ 
(which we shall identify with $\aut(Y_\sing)$) and a $2$-torus.  
The group of connected components $\pi_0(\PSL(V)_Y)$ is a direct product
$\aut(Y_\sing)\times \PSL(V'')_C$.
\end{lemma}
\begin{proof}
Any projective automorphism $g$ of $Y$ preserves its singular set, and hence the plane
spanned by that set (which is defined by putting each $T_i$ equal to zero). The action on that plane preserves the coordinate triangle
defined by $S_0S_1S_2$. The homomorphism $\SL (V')\subset \SL (V')\times \SL(V'')\subset \SL (V)\to \PSL(V)$ is injective and via this embedding we see  
the stabilizer of $S_0S_1S_2$ in $\SL (V')$ faithfully act on $Y$.
If $g$ acts trivially on this plane, then we can represent $g$ by a transformation $\tilde g\in \GL(V)$ such that $\tilde g^*S_i-S_i$ and $\tilde g^*T_i$ are 
linear combinations of $T_0,T_1,T_2$. Since $\tilde g$ multiplies
$S_0S_1S_2-\Phi (T_0,T_1,T_2)$ by a scalar, we see that we must have 
$\tilde g^*S_i=S_i$ for all $i$. It then follows that $\tilde g^*\Phi=\Phi$.
The lemma now follows easily.
\end{proof}

Choose a marking of  $Y_t$. This identifies $H_1(C)\cong I$ with
a primitive isotropic sublattice of $\Lambda_o$. This identification
equips $I$ with a natural orientation and that makes it determine a rational boundary component of the disconnected domain $\DD$ (and hence singles out a component of $\DD$ as well). It is known from singularity theory that the local monodromy group  of $Y$ (in the space of all cubics) is the subgroup of $\G$ generated by the reflections in the (long) roots in $I^\perp$. We therefore denote it by $W_{I^\perp}$. It is clearly a normal 
subgroup of $Z_\G (I)$, the group of $\g\in\G$ that leave $I$ pointwise fixed.
 
Any projective automorphism $g$ of $Y$ sends a smooth cubic $Y'$ near $Y$ to another such $g(Y')$. A path in the space of smooth cubics near $Y$ identifies 
$H_4(g(Y'))$ with $H_4(Y')$. Another path makes the two identifications differ
by an element of $W_{I^\perp}$. Thus we get a group homomorphism 
$\PGL(V)_Y\to Z_\G (I)/W_{I^\perp}$. This clearly factors through $\pi_0(\PGL(V)_Y)$.

\begin{lemma}\label{lemma:autos}
The resulting homomorphism $\pi_0(\PGL(V)_Y)\to Z_\G (I)/W_{I^\perp}$ is an isomorphism.
\end{lemma}
\begin{proof}
The proof amounts to a careful comparison between $W_{I^\perp}$ and $Z_\G (I)$.
The  group $W_{I^\perp}$  acts in
$I^\perp/I$ as the (finite) Weyl group $W_Q$ of the roots in $I^\perp/I$ (which
is isomorphic to $W(E_6)^3$). The kernel $H_Q$ of $W_{I^\perp}\to W_Q$ is a Heisenberg group: the image of $H_Q$ in $\aut(I^\perp)$ is the abelian group 
$Q\otimes I$ and the central extension is by the infinite cyclic group $\det (I)$ and given by the obvious antisymmetric map  $(Q\otimes I)\times (Q\otimes I)\to I\wedge I$
(see section (1.1) of \cite{looijenga}). We have thus described a filtration 
\[
\det (I)\subset H_Q\subset W_{I^\perp}
\]
with successive quotients $Q\otimes I$ and $W_Q$. If we do similarly for $Z_\G (I)$, we get
\[
\det (I)\subset H_{I^\perp/I}\subset Z_\G (I)
\]
with successive quotients $(I^\perp/I) \otimes I$ and $\orth(I^\perp/I)$. 
The quotient $(I^\perp/I)/Q$ is cyclic of order $3$ and if we identify
this group with $\mu_3$, then $(I^\perp/I) \otimes I/ Q\otimes I$ gets identified 
with $\mu_3\otimes I\cong H_1(C,\mu_3)$. This is accounted for by the translations
of order  3 in $C$ that sit in $\pi_0(\PGL(V)_Y)$. According to Remark \ref{rem:discr},  $\orth(I^\perp/I)/W_Q$ can be identified with the product 
$\{ \pm 1\}\times \aut(Y_\sing)$. The first factor  is accounted for by the involution in
$C$ and the second is in place already. This completes the proof.
\end{proof}

\begin{proof}[Proof of Proposition \ref{prop:degone}]
The $\GL (V)$-orbit of $F$ is determined by the $j$-invariant of $C$. The union 
$\Ocal\subset\sym^3V^*$ of such orbits (so with varying $\Phi$) is of codimension
$21$ and an affine-linear section to $\Ocal$ at  $F$ is obtained as follows: the ideal $J(\Phi)$ in $\CC[T_0,T_1,T_2]$ spanned by the partial derivatives of $\Phi$ is of codimension $8$ and the jacobian algebra $\CC[T_0,T_1,T_2]/J(\Phi)$ is graded with summands in degree 0,1,2,3 of dimension 1,3,3,1 respectively. The 
$\SL(V'')$-stabilizer  of $\Phi$ is finite and acts on $J(\Phi)$.
Let $J_0\oplus J_1\oplus J_2\subset \CC[T_0,T_1,T_2]$ be a graded linear lift of the first three summands (so of dimension 7) that is equivariant with respect to the finite group
$\tilde H_3$ (which stabilizes $\Phi$), put $N_i:=\sum_{k=1}^3 S_i^kJ_{3-k}$ and $N:=\sum_{i=0}^2N_i$. Then $F+N$ is a linear section to $\Ocal$ that
is invariant under the group $\aut(Y_\sing)\times \tilde H_3$. 

The summand  $N_1+N_2$ describes deformations of $F_\phi$ that do not affect
the analytic type of the singularity in $p_0$: the local equation in $s_1s_2+\Phi (t_0,t_1,t_2)$ is  altered by a homogeneous polynomial of degree 3 that lies in
$\sum_{i=1}^2\sum_{k=1}^3 s_i^kJ_{3-k}$ (where we view $J_{3-k}$ as a subspace of $\CC[t_0,t_1,t_2]$). It is well-known from singularity theory (splitting
of squares) that such deformations do not change the analytic type. On the other hand, the summand $N_0$ yields almost the full semi-universal deformation of the  simple elliptic singularity $p_0$: we deform in $J_0\oplus J_1\oplus J_2\subset \CC[t_0,t_1,t_2]$ and thus get a codimension one subspace transversal to the equisingularity stratum. We shall identify $N_0$ with this deformation space. We do likewise for the other cases.

At this point we need to recall our work on the deformation theory of the simple elliptic singularities \cite{looijenga:se}. Let $G= G_0+G_1+G_2\in N$ with $G_i\in N_i$ be such that the the cubic $Y_G$ defined by $F+G=0$ is nonsingular. For simplicity, we assume that each $G_i$ is close to zero. Then $p_i$ has a Milnor fiber $M_i\subset 
Y_{F+G}$. These Milnor fibers are pairwise disjoint. We have $H_4(M_i)$ that  is free of rank
$8$. The intersection form on $H_4(M_i)$ has a rank $2$ kernel and the residual lattice is isomorphic
to the root lattice of type $E_6$. If we merely know that each $G_i$ is nonzero, then the singular points of $Y_G$ are all of type $A$, $D$ or $E$ and lie in the (possibly singular) Milnor fiber $M_i$. In \cite{looijenga:se} we considered the period mapping 
that essentially assigns to $G= G_0+G_1+G_2$ with $G_i\not=0$ for all $i$ the periods of $\omega_{F+G}$ on $H_4(M_i)$. We proved there a rather precise Torelli type of result, which may be stated as follows: if $\DD(I)\subset \DD^b$ denotes the rational boundary component defined by the oriented plane $I$, and $U$ is a small   transversal slice to $F$ in $\Ocal$, then the period map defines a map 
\[
U\times N\to W_{I^\perp}\bs \Star(S,\DD^\bb)
\] 
that is a local isomorphism along $U$.  The map 
$W_{I^\perp}\bs  \Star(S,\DD^\bb)\to \G\bs \DD^\bb=X$ factors through 
$Z_\G (I)\bs \Star(S,\DD^\bb)$
and the Baily-Borel construction shows that near a generic point of $\DD(I)$, the map  $Z_\G (I)\bs \Star(S,\DD^\bb)\to X$ is a local isomorphism. 
On the other hand, the  composite map
\[
N\to W_{I^\perp}\bs \Star(S,\DD^\bb)\to Z_\G (I)\bs \Star(S,\DD^\bb)
\]
has, according to Lemma \ref{lemma:autos}, the property that
any fiber  is contained in a $\PGL(V)_Y$-orbit and hence in a $\PGL(V)_Y$-orbit. 
Similarly, we have that for any other $u\in U$, a fiber of 
$\{ u\}\times N\to W_{I^\perp}\bs \Star(S,\DD^\bb)\to Z_\G (I)\bs \Star(S,\DD^\bb)$
is contained in a $\PGL(V)$-orbit. It follows that the period map separates the orbits near $Y$.
\end{proof}

\end{document}